\documentclass[usenames,dvipsnames,a4paper,reqno]{amsart}
\usepackage[utf8]{inputenc}
\usepackage{preamble}
\usepackage{semtkzX}

\title{Local Solubility of Hyperelliptic Curves}
\author{Omri Faraggi}
\address{Department of Mathematics, University College London, London WC1H 0AY, UK}
\email{omri.faraggi.17@ucl.ac.uk}
\date{February 2021}

\begin{document}

	\begin{abstract}
		We give a condition for a hyperelliptic curve $C$ over a local field $K$ to be locally soluble, on the condition that $C$ obtains semistable reduction after a tame extension of $K$, and that the residue field $k$ is sufficiently large relative to the genus of the curve. The condition is presented in terms of the cluster picture of $C$, a combinatorial object which determines much of the local arithmetic of $C$.
	\end{abstract}

	\maketitle
	\renewcommand{\contentsname}{Table of contents}
	\thispagestyle{empty}\setcounter{tocdepth}{1}
	\tableofcontents

\section{Introduction}

Let $C:y^2=f(x)$ be a hyperelliptic curve of genus $g \geq 2$ over a local field $K$ with residue $k$ of characteristic $p>2$. In this note we give a condition for $C$ to have a $K$-rational point in terms of the \textit{cluster picture} of $C$, given that $C$ has tame reduction --- in other words, it obtains semistable reduction after a tame extension of $K$. The cluster picture $\Sigma$ of $C$ is a combinatorial object which contains the data of the $p$-adic distances between the roots of the defining polynomial $f$. This removes dependency on the exact nature of $f$, and cluster pictures have been used used to calculate the Galois representation, semistable model, conductor and minimal discriminant of $C$ in \cite{DDMM18}, the minimal snc model if $C$ has tame reduction in \cite{faraggi20models}, the Tamagawa number in \cite{betts2018computation}, the root number in \cite{bisatt2019clusters} and finally differentials in \cite{kunzweiler2019differential} and \cite{muselli2020models}. A survey article is available at \cite{best2020users}. Since our condition is in terms of the cluster picture of $C$, it allows us to determine whether families of hyperelliptic curves are locally soluble, rather than having to consider individual curves. With some additional work, this should allow us to determine what proportion of hyperelliptic curves of a given genus are locally soluble. This follows on from work such as \cite{bhargava2020proportion} and \cite{cremona2020local}.

Throughout the note the hyperelliptic curve $C$ will always be defined by an equation $y^2 = f(x)$ and $\Rcal$ will denote the set of roots of $f$. In particular, it cannot be a degree $2$ cover of a $\PP^1$ with no point. The absolute Galois group $\GK = \Gal(\overline{K}/K)$ acts on the roots of $\Rcal$ and hence there is an induced action on the cluster picture of $C$. 

\section{The Condition}

Here we state the main result of this note. A knowledge of cluster pictures is required, and so a brief exposition for the uninitiated is available in Appendix \ref{app:clusters}; for more details the reader is encouraged to turn to \cite{DDMM18} or \cite{best2020users}.

\begin{lemma}
	\label{lem:hassweil}
	Let $C:y^2 = f(x)$ be a hyperelliptic curve with tame reduction over a local field $K$ and let $\mathcal{C}$ be a regular snc model of $C$. Suppose that the residue field of $K$ has size $q > 2(g(C)^2 - 1)$. Then any smooth component $\Gamma$ of the special fibre which is fixed by Frobenius has a smooth $k$-rational point.
\end{lemma}
\begin{proof}
	Assume $f$ has odd degree, since the even case is dealt with similarly. Note that any smooth component fixed by Frobenius must be defined over $k$. The Hasse-Weil bound then states that $\left|\# \Gamma(k) - (q+1)\right| \leq 2g(\Gamma)\sqrt{q}$, and hence $\Gamma(k)$ is non-empty if $q > 2(g(\Gamma)^2 - 1)$. Since $g(\Gamma) \leq g(C)$, we are done if $\mathcal{C}$ is a smooth model.
	
	Now assume $\mathcal{C}$ is a semistable model. If $\Gamma$ is genus $0$ then it has $k$-points, so assume it has positive genus and comes from a principal cluster $\s$. Let $I$ be the number of intersection points of other components. Then $\Gamma$ has a smooth $k$-rational point if $q - I > 2(g(\Gamma)^2 - 1)$, and hence it has a smooth $k$-point if $2g(C)^2 - I > 2g(\Gamma)^2$. Suppose $\s = \Rcal$ is odd, and that $\s$ has $s_s$ children of size $1$, $s_o$ proper odd children and $s_e$ proper even children. Then by \cite[Theorem~8.5]{DDMM18}, $I = s_o + 2s_e$ and $2g(\Gamma) + 1 = s_o + s_s$. We also know that $2g(C) + 1 =\deg(f) \geq s_s + 2s_e + 3s_o$. Putting these together the result follows, and similarly if $\s$ is even or $\s \neq \Rcal$
	
	Now assume $C$ has tame reduction. A similar argument works using \cite[Theorem~7.12,7.18]{faraggi20models}, noting that since $\Gamma$ is a smooth component it has $e_{\s} = 1$ and its only intersection points are loops or linking chains to other principal components.
\end{proof}

\begin{prop}
	\label{prop:ratpoints}
	Let $X$ be a curve over a local field $K$ and let $\X$ be a regular snc model of $X$. Then $X$ has a $K$-rational point if and only if $\X_k$ has a component of multiplicity $1$ which is fixed by Frobenius and has a smooth $k$-rational point.
\end{prop}
\begin{proof}
	By \cite[Corollary~9.1.32]{liu2006algebraic} there is a reduction map $\mathrm{red}: X(K) \rightarrow \X_k(k)$ landing in the smooth locus of $\X_k$, which is onto since $K$ is henselian. Therefore $X(K)$ is empty if and only if the smooth locus of $\X_k(k)$ is empty. Since the smooth locus of $\X_k(k)$ consists of points lying on components of multiplicity $1$ which are fixed by Frobenius, the result follows.
\end{proof}

\begin{thm}
	\label{thm:multonefrob}
	Let $C$ be a hyperelliptic curve with tame reduction over a local field $K$ and let $\Ccal$ be the minimal snc model of $C$. If $\Rcal$ is principal, then $\Ccal_k$ has a component of multiplicity $1$ fixed by Frobenius precisely if at least one of the follow occurs:
	\begin{enumerate}
		\item there is a principal cluster $\s$ fixed by $\GK$ with $e_{\s} = 1$, and if in addition $\s$ is \"ubereven, the character $\epsilon_{\s}$ is trivial on $\GK$;
		\item there is a principal cluster $\s$ fixed by $\GK$ (and $\epsilon_{\s}$ trivial if $\s$ \"ubereven) with $e_{\s} > 1$ and at least one of the following:
		\begin{enumerate}
			\item $\s = \Rcal$ and either $\Rcal$ is odd or $\Rcal$ is even and $\epsilon_{\Rcal}$ is the trivial character,
			\item $\s$ has a stable child of size $1$ or $g(\s) = 0$, $\s$ is not \"ubereven and $\s$ has no proper stable odd child, 
			\item $\s$ has no stable proper child, $\lambda_{\s} \in \Z$, $v_K(c_{\s})$ is even and either $g(\s) > 0$ or $\s$ is \"ubereven,
			\item the children of size $1$ of $\s$ are fixed by $\GK$;
		\end{enumerate}
		\item there is a pair of principal clusters $\s' < \s$, both fixed by $\GK$, either with $\s'$ odd and $[-\lambda_{\s} - \delta_{\s'}/2, -\lambda_{\s}] \cap \Z \neq \emptyset,$ or $\s'$ even, the character $\epsilon_{\s'}$ trivial on $\GK$ and $[-d_{\s'}, -d_{\s}]\cap \Z \neq \emptyset$;
		\item there is a twin $\tfrak$ fixed by $\GK$ and either:
		\begin{enumerate}
			\item the character $\epsilon_{\tfrak}$ is trivial and $[-d_{\tfrak}, -d_{P(\tfrak)}]\cap \Z \neq \emptyset$,
			\item the character $\epsilon_{\tfrak}$ is trivial on inertia, $\epsilon_{\tfrak}(\Frob) = -1$, $d_{\tfrak} \in \Z$ and $\nu_{\tfrak} \in 2\Z$ or,
			\item the character $\epsilon_{\tfrak}$ is non-trivial on inertia and $v_K(c_{\tfrak})$ even;
		\end{enumerate} 
	\end{enumerate}
	If $\Rcal$ is not principal, then $\Ccal_k$ has a component of multiplicity $1$ fixed by Frobenius in the cases above, and the additional following cases:
\begin{enumerate}
	\setcounter{enumi}{4}
	\item there is a cotwin $\s < \tfrak$ fixed by $\GK$ and either:
	\begin{enumerate}
		\item the character $\epsilon_{\tfrak}$ is trivial and $[-d_{\s}, -d_{\tfrak}]\cap \Z \neq \emptyset$,
		\item the character $\epsilon_{\tfrak}$ is trivial on inertia, $\epsilon_{\tfrak}(\Frob) = -1$, $d_{\tfrak} \in \Z$ and $\nu_{\tfrak} \in 2\Z$ or,
		\item the character $\epsilon_{\tfrak}$ is non-trivial on inertia and trivial on Frobenius;
	\end{enumerate} 
	\item the top cluster $\Rcal = \s_1 \sqcup \s_2$ is not principal and either:
	\begin{enumerate}
		\item $\s_1$ is odd, fixed by $\GK$ and $[-\lambda_{\Rcal} - \delta_{\s_1}/2, -\lambda_{\Rcal}] \cap \Z \neq \emptyset$,
		\item $\s_1$ is odd, fixed by inertia but not Frobenius and $d_{\Rcal} \in \Z$,
		\item $\s_1$ is odd, inertia swaps $\s_1$ and $\s_2$ and $\epsilon_{\Rcal}(\Frob) = 1$,
		\item $\s_1$ is even, fixed by $\GK$, $\epsilon_{\s_1}$ is trivial and $[-d_{\s_1}, -d_{\Rcal}] \cap \Z \neq \emptyset$,
		\item $\s_1$ is even, fixed by inertia but not Frobenius, $\epsilon_{\s_1}$ is trivial and $d_{\Rcal} \in \Z$, 
		\item $\s_1$ is even, inertia swaps $\s_1$ and $\s_2$, $\epsilon_{\s_1}$ is trivial and $\epsilon_{\Rcal}(\Frob) = 1$.
	\end{enumerate}
	\end{enumerate}
\end{thm}
\begin{proof}
	For brevity, call components of multiplicity $1$ which are fixed by Frobenius \textit{good}. If $\Ccal$ has a good component, it must either by a principal component of part of a chain of rational curves. We investigate these cases separately, beginning with principal components. The principal components are parametrised by principal clusters\footnote{Except in crossed tails but these always have even multiplicity.}, and by \cite[Theorem~7.12]{faraggi20models} a good component must come from a principal cluster $\s$ fixed by $\GK$. In addition, $\s$ must have $e_{\s} = 1$, i.e. $d_{\s} \in \Z$ and $\nu_{\s} \in 2\Z$, and if $\s$ is \"ubereven, the character $\epsilon_{\s}$ must be trivial. In all these cases we do indeed have a good component. This is case (i).
	
	We are left to find the cases where there is a good component in a chain of rational curves. Chains of rational curves appear in several flavours: linking chains between principal components, loops from a principal component to itself, and tails (including the crosses of crossed tails). A principal cluster $\s$ fixed by Galois with $e_{\s} > 1$ contributes tails to the special fibre. If $\s = \Rcal$ then the $\infty$-tails will have a good component if $\Rcal$ is odd or if $\Rcal$ is even and the character $\epsilon_{\Rcal}$ is trivial by the first three rows of the first table in \cite[Theorem~7.18]{faraggi20models}, noting that $\epsilon_{\Rcal}(\Frob)$ swaps the $\infty$-tails if and only if it is $-1$ by Theorem \cite[Theorem~7.21]{faraggi20models}. If $\s$ is a general principal cluster then a $(0,0)$-tail exists if (ii)(b) is satisfied and it will always have a good component An $(x_{\s}=0)$-tail will have a good component so long as $\lambda_{\s} \in \Z$ and $v_K(c_{\s})$ is even --- the former ensures that there are two $(x_{\s}=0)$-tails by Theorem \cite[Theorem~7.18]{faraggi20models} and the latter that they are fixed by Frobenius, by  \cite[Theorem~7.21]{faraggi20models}. A $(y_{\s}=0)$-tail will have a good component if and only if the singletons of $\s$ are fixed by Galois. This is case (ii).
	
	Suppose $L$ is a linking chain arising from a pair of orbits $X' < X$. By \cite[Theorem~7.18]{faraggi20models}, the lowest common multiple of multiplicities of the components of $L$ is $|X'|$, so $X = \s$, $X' = \s'$ must be clusters fixed by Galois. If $\s'$ is odd then $L$ has a good component if and only if $[-\lambda_{\s}-  \delta_{\s'}/2, -\lambda_{\s}] \cap \Z \neq \emptyset$. If $\s'$ is even then $L$ has a good component if and only if $\epsilon_{\s'}$ is trivial and $[-d_{\s'}, -d_{\s}]\cap \Z \neq \emptyset$. This is because $\sigma \in \GK$ swaps the two linking chains connecting $\Gamma_{\s,L}$ and $\Gamma_{\s',L}$ if and only if $\epsilon_{\s'}(\sigma) = -1$ This is case (iii). 
	
	Now suppose $L$ is a loop arising from an orbit of twins. By the same argument as above, this must in fact arise from a twin $\tfrak$ fixed by Galois with $\epsilon_{\tfrak}$ trivial on inertia. If $\epsilon_{\tfrak}$ is further trivial on Frobenius then $L$ will have a good component if and only if it has a component of multiplicity $1$, which occurs precisely when $[-d_{\tfrak}, -d_{P(\tfrak)}]\cap \Z \neq \emptyset$. If $\epsilon_{\tfrak}(\Frob) = -1$ then Frobenius inverts the loop, but there can still be a component fixed by Frobenius if $L$ is a loop of odd length. This happens precisely when $d_{\tfrak}$ is an integer.  
	
	If $T$ is instead a crossed tail, it must arise from a twin $\tfrak$ fixed by Galois with $\epsilon_{\tfrak}$ non-trivial on inertia. In this case, the crosses of $T$ have multiplicity $1$, and by   they are fixed by Frobenius if and only if $v_K(c_{\tfrak})$ is even. This and the previous paragraph is case (iv).
	
	Cases (v) and (vi), where $\Rcal$ is a non-principal cluster, are dealt with similarly.
	
\end{proof}

\begin{cor}
	Let $C$ be a hyperelliptic curve with tame reduction over a local field $K$ with residue field $k$ of characteristic $p>2$. Suppose that $|k|=q$ is such that $q > 2(g(C)^2 - 1)$. Then $C$ has a $K$-rational point in precisely the cases described in Theorem \ref{thm:multonefrob}.
\end{cor}
\begin{proof}
	By Lemma \ref{lem:hassweil}, any smooth components of the special fibre fixed by Frobenius have a $k$-point. By Theorem \ref{thm:multonefrob}, a smooth component fixed by Frobenius exists and hence by Proposition \ref{prop:ratpoints} $C$ has a $K$-rational point.
\end{proof}

\section{Examples} We give some examples to illustrate our theorem. All curves are over $\Q_p$, with $p$ sufficiently large so that all multiplicity $1$ components of the special fibre which are fixed by Frobenius have points.

\begin{eg}
	Let $C:y^2 = (x^4 - p^{17})(x^3 - p^2).$ We immediately observe that $\Rcal$ is principal, odd and $e_{\Rcal} > 1$, so $C$ must have a $K$-rational point by condition (ii)(a).
\end{eg}

\begin{eg}
	Let $C:y^2 = p((x-1)^2 + p^2)((x-\zeta_3)^2 +p^2)((x-\zeta_3^2)^2 + p^2),$ with $p \equiv -1 \, \mathrm{mod} \, 3$ and $\zeta_3$ a fixed cube root of unity. The cluster picture of $C$ is shown in Figure \ref{fig:eg1}. There is a unique principal cluster $\Rcal$, which is \"ubereven. However, $e_{\Rcal}=2$ and $\epsilon_{\Rcal}(\sigma) = (-1)^v_K(c_f) = -1$ for $\sigma$ a generator of $\Gal(\Q_p(\sqrt{2})/\Q_p)$ and so $\epsilon_{\Rcal}$ is non-trivial. Therefore condition (i) and (ii) are not satisfied. Since there is only $1$ principal cluster (iii) is not satisfied. Therefore we are left to check (iv). But the three twins are permuted by Frobenius and hence (iv) is not satisfied. Therefore $C$ has no $K$-rational points. This can also be seen from looking at the minimal snc model --- the components of multiplicity $1$ are the crosses of the three crossed tails. But the crossed tails are permuted by Frobenius as the corresponding clusters are, so there is no component of multiplicity $1$ fixed by Frobenius.
	\begin{figure}[ht]
		\centering
		\clusterpicture
		\Root[] {} {first} {r1};
		\Root[] {} {r1} {r2};
		\ClusterLDName c1[][1][\tfrak_1] = (r1)(r2);
		\Root[] {} {c1} {r3};
		\Root[] {} {r3} {r4};
		\ClusterLDName c2[][1][\tfrak_2] = (r3)(r4);
		\Root[] {} {c2} {r5};
		\Root[] {} {r5} {r6};
		\ClusterLDName c3[][1][\tfrak_3] = (r5)(r6);
		\ClusterLDName c4[][0][\Rcal] = (c1)(c2)(c3);
		\endclusterpicture \quad \raisebox{-2.5em}{\begin{tikzpicture}
			\tikzstyle{every node}=[font=\tiny]
			\draw[very thick] (0, 0) -- ++(0, 2) node[left] {2};
			\draw (-0.2, 1.5) -- ++ (1, 0.5) node[above] {2}; 
			\draw (0.4, 1.9) -- ++ (1.2, 0) node[right] {2}; 
			\draw (1.2, 2.1) -- ++ (0, -0.4);
			\draw (1.4, 2.1) -- ++ (0, -0.4);
			\draw (-0.2, 0.9) -- ++ (1, 0.5) node[above] {2}; 
			\draw (0.4, 1.3) -- ++ (1.2, 0) node[right] {2}; 
			\draw (1.2, 1.5) -- ++ (0, -0.4);
			\draw (1.4, 1.5) -- ++ (0, -0.4);
			\draw (-0.2, 0.3) -- ++ (1, 0.5) node[above] {2}; 
			\draw (0.4, 0.7) -- ++ (1.2, 0) node[right] {2}; 
			\draw (1.2, 0.9) -- ++ (0, -0.4);
			\draw (1.4, 0.9) -- ++ (0, -0.4);
			\end{tikzpicture}}
		\caption{Cluster picture and minimal snc model of $C:y^2 = p((x-1)^2 + p^2)((x-\zeta_3)^2 +p^2)((x-\zeta_3^2)^2 + p^2).$}
		\label{fig:eg1}
	\end{figure}
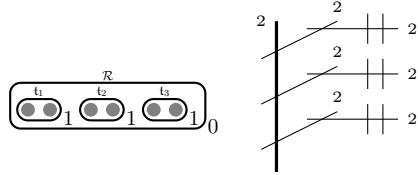
\end{eg}

\begin{eg}
	Let $C:y^2 = p(x^3 - p^2)((x-1)^3 - p^2).$ This is a hyperelliptic curve of Namikawa-Ueno type $\mathrm{II}^*-\mathrm{II}^*-\alpha$. The cluster picture of $C$ is shown in Figure \ref{fig:eg2}. The two principal clusters $\s_1$ and $\s_2$ have $e_{\s_i} = 6$ and so (i) does not apply. Quick inspection reveals that (ii)-(v) don't either. Continuing, we see that the top cluster $\Rcal = \s_1 \sqcup \s_2$ is not principal and $\s_1$ is odd and fixed by Galois. However, $[-\lambda_{\Rcal} - \delta_{\s_1}/2, -\lambda_{\Rcal}] = [-11/6, -3/2]$ which does not contain an integer. Therefore condition (vi) does not give us a $K$-rational point, and therefore $C$ cannot have any $K$-rational points. This can also be seen from the minimal snc model as there is no component of multiplicity $1$.
	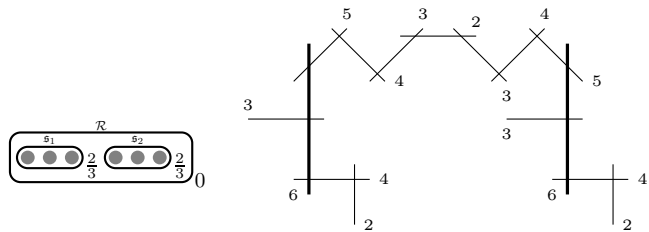
\begin{figure}[ht]
		\centering
		\clusterpicture
		\Root[] {} {first} {r1};
		\Root[] {} {r1} {r2};
		\Root[] {} {r2} {r3};
		\ClusterLDName c1[][\frac23][\s_1] = (r1)(r2)(r3);
		\Root[] {} {c1} {r4};
		\Root[] {} {r4} {r5};
		\Root[] {} {r5} {r6};
		\ClusterLDName c2[][\frac23][\s_2] = (r4)(r5)(r6);
		\ClusterLDName c2[][0][\Rcal] = (c1)(c2);
		\endclusterpicture \quad
	\raisebox{-3em}{\begin{tikzpicture}
	\tikzstyle{every node}=[font=\tiny]
	\draw[very thick] (-1.2,-0.7) -- ++(0,-2) node[left] {6};
	\draw[very thick] (2.2, -0.7) -- ++(0, -2) node[left] {6};
	
	\draw (-1, -1.7) -- ++(-1, 0) node[above] {3};
	\draw (-1.4, -2.5) -- ++(1, 0) node[right] {4};
	\draw (-0.6, -2.3) --++(0, -0.8) node[right] {2};
	
	\draw (2.4, -1.7) -- ++(-1, 0) node[below] {3};
	\draw (2, -2.5) -- ++(1, 0) node[right] {4};
	\draw (2.8, -2.3) --++(0, -0.8) node[right] {2};

	\draw (-1.4, -1.2) -- ++(0.7, 0.7) node[above] {5};
	\draw (-0.9, -0.5) -- ++(0.7, -0.7) node[right] {4};
	\draw (-0.4, -1.2) -- ++(0.7, 0.7) node[above] {3};
	\draw (0, -0.6) -- ++(1, 0) node[above] {2};
	\draw (0.7, -0.5) -- ++(0.7, -0.7) node[below] {3};
	\draw (1.2, -1.2) -- ++(0.7, 0.7) node[above] {4};
	\draw (1.7, -0.5) -- ++(0.7, -0.7) node[right] {5};
	
	\end{tikzpicture}}
	
	\caption{Cluster picture and minimal snc model of $C:y^2 = p(x^3 - p^2)((x-1)^3 - p^2).$}
	
	\label{fig:eg2}
\end{figure}
\end{eg}

\appendix
\section{Cluster Pictures}
\label{app:clusters}

We briefly collate some definitions concerning clusters for the convenience of the reader.

\begin{definition}
	\label{def:cluster}
	Let $C : y^2 = f(x)$ be a hyperelliptic curve over $K$, with $\Rcal$ the set of roots of $f$. A \textit{cluster} is a non-empty subset $\s \subseteq \mathcal{R}$ of the form $\s = D \cap \Rcal$ for some disc $D = z + \pi^n \mathcal{O}_K$, where $z \in \overline{K}$ and $n \in \Q$. If $\s$ is a cluster and $|\s| > 1$, $\s$ is a \textit{proper} cluster and we define its \textit{depth} $$d_{\s} = \min_{r, r' \in \s} v_K(r-r').$$
	The \textit{cluster picture} $\Sigma$ is the set of all clusters of the roots of $f$.
\end{definition}

\begin{definition}
	A cluster $\s\in \Sigma_{\chi}$  is \textit{odd} (resp. \textit{even}) if $|\s|$ is odd (resp. even). A \textit{child} $\s' < \s$ is a maximal subcluster of $\s$, and its \textit{relative depth} is $\delta_{\s'} = d_{\s'} - d_{\s}$. The cluster $\s$ is \textit{\"ubereven} if all its children are even. It is a \textit{twin} if $|\s|=2$ and \textit{principal} if $|\s| \geq 3$, except if either $\s=\mathcal{R}$ is even and has exactly two children, or if $\s$ has a child of size $2g$. It is a \textit{cotwin} if it has a child of size $2g$ whose complement is not a twin.
\end{definition}

\begin{definition}
	Let $D$ be a $p$-adic disk. Then $$\nu_D(f) = v_K(c_f) + \sum_{r \in \Rcal} \min(d_D, v_K(z_D - r)),$$ and for a cluster $\s$, $\nu_{\s} = \nu_{D(\s)}$.
\end{definition}

\begin{definition}
	Let $\s$ be a proper cluster. Then $e_{\s} \in \Z$ is the smallest integer such that $e_{\s}d_{\s} \in \Z$ and $e_{\s}\nu_{\s} \in 2\Z$.
\end{definition}

\begin{definition}
	Let $\s \in \Sigma$ be a proper cluster. Then the \textit{genus} of $\s$, $g(\s)$, is such that $2g(\s) + 1$ or $2g(\s) + 2$ equals the number of odd children of $\s$.
\end{definition}

\begin{definition}[Invariants and Characters]
	\label{def:clusterfunctions2}
	For $\sigma\in\GK$ let
	\begin{align*}
		\chi(\sigma) = \frac{\sigma(\pi)}{\pi} \mod \mathfrak{m}.
	\end{align*}
	For a proper cluster $\s\in\Sigma_{C}$ define
	\begin{align*}
		\lambda_{\s} &= \frac{\nu_{\s}}{2} - d_{\s} \sum_{\s'<\s} \bigg\lfloor \frac{|\s'|}{2} \bigg\rfloor.
	\end{align*}
	Define $\theta_\s=\sqrt{c_f\prod_{r\notin \s}(z_\s-r)}.$ If $\s$ is either even or a cotwin, we define $\epsilon_\s: \GK\to\{\pm1\}$ by $$\epsilon_\s(\sigma)\equiv\frac{\sigma(\theta_{\s^*})}{\theta_{(\sigma\s)^*}} \mod \mathfrak{m}.$$
	For all other clusters $\s$ set $\epsilon_\s(\sigma)=0$.
\end{definition}

\begin{definition}
	Let $\s'<\s$ be clusters in $\Sigma$. Then $\s'$ is a \emph{stable child} of $\s$ if the stabiliser of $\s$ also stabilises $\s'$.
\end{definition}

\bibliographystyle{plain}
\bibliography{bibli}

\end{document}